\documentclass[final,1p,times]{elsarticle}

\usepackage{amsthm,amsmath,amssymb,mathrsfs,amsfonts,graphicx,graphics,latexsym,exscale,cmmib57,dsfont,bbm,amscd}
\usepackage{color,soul}

\newtheorem*{thmA}{Theorem A}%
\newtheorem*{thmB}{Theorem B}%

\newtheorem{Theorem}{Theorem}[section]
\newtheorem{B-W}[Theorem]{The Berger-Wang Formula}%
\newtheorem{que}{Question}
\newtheorem{Lemma}[Theorem]{Lemma}

\newtheorem{cor}[Theorem]{Corollary}

\theoremstyle{definition}%
\newtheorem{rem}[Theorem]{Remark}%
\newtheorem{defn}[Theorem]{Definition}
\newtheorem{defn*}[Theorem]{Definitions}%
\newtheorem{example}[Theorem]{\textbf{Example}}%

\numberwithin{equation}{section}

\newcommand{\bS}{\boldsymbol{S}}

\journal{Linear Algebra and its Applications}%

\begin{document}

\begin{frontmatter}


\tnotetext[label1]{Project was supported partly by National Natural Science Foundation of China (No.~11071112).}%

\title{A Gel'fand-type spectral radius formula and stability of linear constrained switching systems\tnoteref{label1}}


\author[]{Xiongping Dai}%

\ead{xpdai@nju.edu.cn}%

\address{Department of Mathematics, Nanjing University, Nanjing 210093, People's Republic of China}%

\begin{abstract}
Using ergodic theory, in this paper we present a Gel'fand-type
spectral radius formula which states that the joint
spectral radius is equal to the generalized spectral radius for a matrix multiplicative semigroup $\bS^+$ restricted to a subset that need not carry the algebraic structure of $\bS^+$. This generalizes the Berger-Wang formula. Using it as a tool,
we study the absolute exponential stability of a
linear switched system driven by a compact subshift of
the one-sided Markov shift associated to $\bS$.
\end{abstract}


\end{frontmatter}


\section{Introduction}\label{sec1}%

In this paper, we study the Gel'fand-type spectral-radius formula and stability of a matrix multiplicative semigroup $\bS^+$ restricted to a subset that does not need to carry the algebraic structure of the semigroup $\bS^+$, using ergodic-theoretic and dynamical systems approaches.

\subsection{The Gel'fand-type formulae}\label{sec1.1}%
Let $d\ge1$ be an integer and $\mathcal{I}$ a metrizable topological space. We consider a continuous matrix-valued function
$\bS\colon\mathcal{I}\rightarrow\mathbb{C}^{d\times d};\,i\mapsto S_{i}$.
Let us denote by $\varSigma_{\!\mathcal{I}}^+$
the set of all the one-sided infinite switching signals $i(\cdot)\colon\mathbb{N}\rightarrow\mathcal{I}$ endowed with the standard infinite-product topology, where $\mathbb{N}=\{1,2,\ldots\}$. For simplicity, we write $i(n)=i_n$ for all $n\in\mathbb{N}$. Then in the state space $\mathbb{C}^d$, we define the linear, discrete-time, switched dynamical system $\bS_{i(\cdot)}$:
\begin{equation*}
x_n=S_{i_{n}}\cdots S_{i_1}x_0\qquad(
x_0\in\mathbb{C}^d,\; n\ge1),
\end{equation*}
for any switching signal $i(\cdot)=(i_n)_{n=1}^{+\infty}\in\varSigma_{\!\mathcal{I}}^+$.
For any word $w=(i_1,\ldots,i_n)\in\mathcal {I}^n=\overset{n\textrm{-time}}{\overbrace{\mathcal{I}\times\cdots\times\mathcal{I}}}$ of
length $n\ge 1$, simply write
$\bS_w=S_{i_n}\cdots S_{i_1}$
and let $\|\bS_w\|$ denote the operator norm of
the linear transformation $x\mapsto
\bS_wx$ induced by any preassigned vector norm
$\|\cdot\|$ on $\mathbb{C}^d$; that is to say,
$\|\bS_w\|=\sup_{x\in\mathbb{C}^d, \|x\|=1}\|\bS_wx\|$.

The \emph{joint spectral radius} of $\bS$ (free of constraints) is
introduced by G.-C.~Rota and G.~Strang in~\cite{RS} as follows:
\begin{equation*}
\hat{\rho}(\bS)=\limsup_{n\to+\infty}\left\{\sup_{w\in\mathcal
{I}^n}\sqrt[n]{\|\bS_w\|}\right\}\quad \left(~=\lim_{n\to+\infty}\left\{\sup_{w\in\mathcal
{I}^n}\sqrt[n]{\|\bS_w\|}\right\}\right).
\end{equation*}
Since
\begin{equation*}
\log\left(\sup_{w\in\mathcal
{I}^{\ell+m}}\|\bS_w\|\right)\le\log\left(\sup_{w\in\mathcal
{I}^m}\|\bS_w\|\right)+\log\left(\sup_{w\in\mathcal
{I}^\ell}\|\bS_w\|\right)
\end{equation*}
for all $\ell, m\ge 1$, i.e., the subadditivity holds, the above limit always exists.
On the other hand, the \emph{generalized spectral radius} of $\bS$ (free of constraints) is defined by
I.~Daubechies and J.C.~Lagarias in~\cite{DL} as
\begin{equation*}
\rho(\bS)=\limsup_{n\to+\infty}\left\{\sup_{w\in\mathcal
{I}^n}\sqrt[n]{\rho(\bS_w)}\right\},
\end{equation*}
where $\rho(A)$ denotes the usual spectral radius of the matrix
$A\in\mathbb{C}^{d\times d}$.

Then, the so-called generalized Gel'fand spectral-radius formula, due to M.A.~Berger and
Y.~Wang~\cite{BW} and conjectured by I.~Daubechies and J.C.~Lagarias~\cite{DL},
can be stated as follows:

\begin{B-W}[See \cite{BW}]\label{thm1.1}%
If $\bS=\{S_i\}_{i\in\mathcal{I}}$ is a bounded subset of $\mathbb{C}^{d\times d}$, then there holds the equality
$\rho(\bS)=\hat{\rho}(\bS)$.
\end{B-W}

This formula was proved by using different approaches,
for example, in~\cite{BW, El, SWP, CZ00, Bochi, Dai-JMAA}. Recently, this formula has been
generalized to sets of precompact linear operators constraint-free acting on
a Banach space by Ian D.~Morris in~\cite{Mor} using ergodic theory.

The above Gel'fand-type spectral-radius formula is an important tool in a
number of research areas, such as in the theory of control and
stability of unforced systems, see \cite{Bar, Koz, Gur, DHX-aut} for example; in coding
theory, see~\cite{MOS}; in wavelet regularity, see \cite{DL, DL92, HS,
Mae}; and in the study of numerical solutions to ordinary
differential equations, see, e.g., \cite{GZ}.

However, in many real-world situations, constraints on allowable switching signals often arise naturally as a result of physical requirements on a system. One often needs to consider
some switching constraints imposed by some kind of uncertainty about the model or
about environment in which the object operates, see \cite{Wir, LD06-1, LD06-2, LD07, BF11} and so on. Consider in the control theory, for example, a proper
subset $\varLambda$ of $\varSigma_{\!\mathcal{I}}^+$ as the
set of admissible switching signals, such as
\begin{equation*}
\varLambda=\varSigma_{\mathbb{A}}^+:=\left\{i(\cdot)=(i_n)_{n=1}^{+\infty}\in\varSigma_{\!\mathcal{I}}^+\,|\,
a_{i_{n}i_{n+1}}=1\ \forall n\ge 1\right\}
\end{equation*}
where $\mathcal{I}=\{1,\ldots,\kappa\}$ consists of finitely many letters and where $\mathbb{A}=(a_{\ell m})$ is a $\kappa\times\kappa$ matrix of zeros
and ones induced by a Markov transition matrix or a directed graph. A more general way to define $\varLambda$ is via a language, as shown, for example in \cite{ZBSB, HMU, LD07}.

So, it is natural and necessary to introduce the definition of Gel'fand-type spectral radius under some switching constraints.

Hereafter, if $\varLambda$ is a nonempty subset of $\varSigma_{\!\mathcal {I}}^+$, then $\bS_{\upharpoonright\!\varLambda}$ is identified with the family of systems $\bS_{i(\cdot)}$ over all switching signals $i(\cdot)\in\varLambda$, and called the switched system with constraint $\varLambda$.

\begin{defn*}
Let $\varLambda$ be a nonempty subset of
$\varSigma_{\!\mathcal {I}}^+$ as the
set of admissible switching signals. Define the \emph{joint spectral radius of $\bS_{\upharpoonright\!\varLambda}$} as
\begin{equation*}
\hat{\rho}(\bS_{\upharpoonright\!\varLambda})=\limsup_{n\to+\infty}\left\{\sup_{i(\cdot)\in\varLambda}\sqrt[n]{\|S_{i_n}\cdots S_{i_1}\|}\right\}.
\end{equation*}
The \emph{generalized spectral radius of $\bS_{\upharpoonright\!\varLambda}$} is defined as
\begin{equation*}
{\rho}(\bS_{\upharpoonright\!\varLambda})=\limsup_{n\to+\infty}\left\{\sup_{i(\cdot)\in\varLambda}\sqrt[n]{\rho(S_{i_n}\cdots S_{i_1})}\right\}.
\end{equation*}
\end{defn*}

We notice that if $\varLambda$ is \emph{invariant} by the natural one-sided Markov shift $\theta_+\colon i(\cdot)\mapsto i(\cdot+1)$; that is, $i(\cdot+1)$ belongs to $\varLambda$ for any $i(\cdot)\in\varLambda$, then from the subadditivity, there follows that
$\hat{\rho}(\bS_{\upharpoonright\!\varLambda})$ is well
defined in the sense that
\begin{equation*}
\hat{\rho}(\bS_{\upharpoonright\!\varLambda})=\lim_{n\to+\infty}\left\{\sup_{i(\cdot)\in\varLambda}\sqrt[n]{\|S_{i_n}\cdots S_{i_1}\|}\right\}.
\end{equation*}
It is easily seen that there holds the inequality
${\rho}(\bS_{\upharpoonright\!\varLambda})\le\hat{\rho}(\bS_{\upharpoonright\!\varLambda})$.
Clearly, $\hat{\rho}(\bS_{\upharpoonright\!\varLambda
})=\hat{\rho}(\bS)$ and
${\rho}(\bS_{\upharpoonright\!\varLambda
})={\rho}(\bS)$ for the special free-constraint case
$\varLambda=\varSigma_{\!\mathcal{I}}^+$, if $\bS$ is bounded in $\mathbb{C}^{d\times d}$.

Based on the recent work of Ian D.~Morris~\cite{Mor} (see Theorem~\ref{thm2.6R} below), in this paper, we present the
following Gel'fand-type spectral-radius formula under switching constraints:

\begin{thmA}[Spectral-radius formula with constraints]
Let $\bS\colon\mathcal{I}\rightarrow\mathbb{C}^{d\times
d};\ i\mapsto S_i$ be continuous in $i\in\mathcal{I}$ where $\mathcal{I}$ is a
metric space, and assume
$\varLambda\subset\varSigma_{\!\mathcal {I}}^+$ is an invariant compact
set of the one-sided Markov shift
\begin{equation*}
\theta_+\colon\varSigma_{\!\mathcal {I}}^+\rightarrow\varSigma_{\!\mathcal
{I}}^+;\quad i(\cdot)=(i_n)_{n=1}^{+\infty}\mapsto i(\cdot+1)=(i_{n+1})_{n=1}^{+\infty}.
\end{equation*}
Then there holds the
equality
${\rho}(\bS_{\upharpoonright\!\varLambda})=\hat{\rho}(\bS_{\upharpoonright\!\varLambda})$.
\end{thmA}

Let
$\bS_{\upharpoonright\!\varLambda}^+$ be the set of all
product matrices $S_{i_n}\cdots S_{i_1}$ where $n\ge1$ and
$i(\cdot)=(i_n)_{n=1}^{+\infty}\in\varLambda$. A technical problem
is, for the constrained case
$\varLambda\subsetneq\varSigma_{\!\mathcal {I}}^+$, that
$\bS_{\upharpoonright\!\varLambda}^+$ does not need to carry the algebraic
structure of a semigroup; otherwise, \cite[Theorem~B]{Bochi} works and implies Theorem~A in our context. The compactness and $\theta_+$-invariance of $\varLambda$ both are needed for our discussion of using ergodic theory.

We note that \cite[Theorem~7.3]{Wir} contains a ``Gel'fand-type formula" with constraints which is for continuous time and in a special case,
using Lyapunov function. Our theorem will be proved in
Section~\ref{sec2} based on a recent theorem of Ian D.~Morris in
\cite{Mor}.

Theorem~A is a generalization of the Berger-Wang formula. In fact, from
it we could obtain concisely the Berger-Wang formula as follows.

\begin{proof}[\textbf{Proof of the Berger-Wang formula}]
Let $\{S_i\,|\,i\in\mathcal{I}\}\subset\mathbb{C}^{d\times d}$ be an arbitrary bounded set.
Write $\mathds{I}=\mathrm{Cl}_{\mathbb{C}^{d\times d}}(\{S_i\,|\,i\in\mathcal{I}\})$, the closure of the set
$\{S_i\colon i\in\mathcal{I}\}$ in $\mathbb{C}^{d\times d}$. Then, $\mathds{I}$ is
compact in $\mathbb{C}^{d\times d}$, and the function
$\mathds{S}\colon\mathds{I}\rightarrow\mathbb{C}^{d\times
d}$, defined by $\mathbf{i}\mapsto \mathds{S}_\mathbf{i}$ where $\mathds{S}_\mathbf{i}=\mathbf{i}\ \forall \mathbf{i}\in\mathds{I}$,
is continuous in $\mathbf{i}\in\mathds{I}$. Since there holds that
\begin{equation*}
\sup_{\mathbf{w}\in\mathds{I}^n}\sqrt[n]{\|\mathds{S}_\mathbf{w}\|}=\sup_{w\in\mathcal{I}^n}\sqrt[n]{\|\bS_w\|}
\quad \textrm{and}\quad
\sup_{\mathbf{w}\in\mathds{I}^n}\sqrt[n]{\rho\left(\mathds{S}_\mathbf{w}\right)}=\sup_{w\in\mathcal{I}^n}\sqrt[n]{\rho(\bS_w)}
\end{equation*}
for all $n\ge 1$ from the fact $\mathds{I}^n=\mathrm{Cl}_{\mathbb{C}^{d\times d}}(\{\bS_w\,|\,w\in\mathcal{I}^n\})$, we can obtain that
$\hat{\rho}\left(\bS\right)=\hat{\rho}\left(\mathds{S}\right)$ and $\rho\left(\bS\right)=\rho\left(\mathds{S}\right)$.
So, applying Theorem~A in the case $\varLambda=\varSigma_\mathds{I}^+$, we have got that $\hat{\rho}\left(\mathds{S}\right)=\rho\left(\mathds{S}\right)$.
This completes the proof of the Berger-Wang formula (Theorem~\ref{thm1.1}).
\end{proof}

In addition, we define the Lyapunov exponent associated to an initial state $x_0\in\mathbb{C}^d\setminus\{\mathbf{0}\}$ and a switching signal $i(\cdot)=(i_n)_{n=1}^{+\infty}$ by
\begin{equation*}
\chi(x_0,\bS_{i(\cdot)})=\limsup_{n\to+\infty}\frac{1}{n}\log\|S_{i_n}\cdots S_{i_1}x_0\|.
\end{equation*}
It is easily seen that
$\hat{\rho}\left(\bS\right)\ge\exp\chi(x_0,\bS_{i(\cdot)})$
for all $i(\cdot)\in\varLambda$ and all $x_0\in\mathbb{C}^d$.
However, we will prove that $\hat{\rho}\left(\bS\right)$ might
be achieved by some optimal pair $(x_0,i(\cdot))\in\mathbb{C}^d\times\varLambda$;
see Corollary~\ref{cor2.7R} below, which generalizes a corresponding result in \cite{Bar} in the free-constraints case.

Recall for any given $i(\cdot)\in\varSigma_{\!\mathcal{I}}^+$ that $\bS$ is said to be $i(\cdot)$-exponentially stable, provided that there exists $\mathbf{c}\ge1$ and $\chi<0$ such that
\begin{equation*}
\|S_{i_n}\cdots S_{i_1}x_0\|\le \mathbf{c}\|x_0\|\exp(n\chi)\quad\forall x_0\in\mathbb{C}^d \textrm{and }n\ge1.
\end{equation*}
This is equivalent to
\begin{equation*}
\chi(\bS_{i(\cdot)}):=\limsup_{n\to+\infty}\frac{1}{n}\log\|S_{i_n}\cdots S_{i_1}\|<0.
\end{equation*}
Moreover, this is also equivalent to $\chi(x_0,\bS_{i(\cdot)})<0$ for all $x_0\in\mathbb{C}^d\setminus\{\mathbf{0}\}$.
Further, $\bS$ is called to be \emph{uniformly} $i(\cdot)$-exponentially stable, provided that there exists $C\ge1$ and $\chi<0$ such that
\begin{equation*}
\|S_{i_{m+\ell}}\cdots S_{i_\ell}\cdots S_{i_1}x_0\|\le C\|S_{i_\ell}\cdots S_{i_1}x_0\|\exp(m\chi)\quad\forall x_0\in\mathbb{C}^d \textrm{and }m\ge1,
\end{equation*}
uniformly for $\ell\ge0$. This is equivalent to that $\bS$ is exponentially stable over the closure of the orbit $\{i(\cdot+m)\colon m=0,1,2,\ldots\}$ in $\varSigma_{\!\mathcal{I}}^+$.

From \cite{DHX-aut} together with K.G.~Hare \textit{et al}.~\cite{HMST}, one can construct an explicit counterexample to show that the $i(\cdot)$-exponential stability is essentially weaker than the uniform $i(\cdot)$-exponential stability of $\bS$.
\subsection{Stability criteria under switching-path constraints}\label{sec1.2}%
As pointed out in D.~Liberzon and A.S.~Morse~\cite{LM99}, there are three benchmark problems for switched systems: stabilization under arbitrary switching signals, stabilization under a switching path constraint, and construction of stabilizing switching signals.
To the second problem, as another result of our spectral-radius formula, in the second part of
this paper, we give the following criteria of the absolutely asymptotic
stability for a linear system obeying switching constraints, which
will be proved in Section~\ref{sec3}.

\begin{thmB}
Let $\bS\colon\mathcal{I}\rightarrow\mathbb{C}^{d\times
d}$ be continuous and bounded with $\rho(\bS)=1$ and assume
$\varLambda\subset\varSigma_{\!\mathcal {I}}^+$ is an invariant compact
set of the one-sided Markov shift
$\theta_+\colon\varSigma_{\!\mathcal{I}}^+\rightarrow\varSigma_{\!\mathcal{I}}^+$.
Then, the following conditions are mutually equivalent:
\begin{description}
\item[(a)] $\bS$ is ``$\varLambda$-absolutely
asymptotically stable", i.e.,
\begin{equation*}
S_{i_n}\cdots S_{i_1}\to \mathbf{0}_{d\times d}\; \textrm{ as }n\to+\infty
\quad\forall i(\cdot)\in\varLambda,
\end{equation*}
where $\mathbf{0}_{d\times d}$ is the origin of $\mathbb{C}^{d\times d}$.

\item[(b)] The generalized spectral radius $\rho(\bS_{\upharpoonright\!\varLambda})<1$.

\item[(c)] There exists a constant $0<\gamma<1$ and an integer $N\ge1$ such that
\begin{equation*}
\rho(S_{i_n}\cdots S_{i_1})\le\gamma\quad \forall n\ge N \textrm{ and } i(\cdot)\in\varLambda.
\end{equation*}
\end{description}
\end{thmB}

The claim $\mathrm{(a)}\Leftrightarrow\mathrm{(b)}$ still holds without the assumption $\rho(\bS)=1$, by using the Fenichel uniformity theorem (Lemma~\ref{lem3.3} below) and Theorem~A; see Lemmas~\ref{lem3.2} and \ref{lem3.3} below.
Here the compactness of $\varLambda$ is important for the proof of Theorem~B presented in this paper. Let us see a simple counterexample as follows:

\begin{example}\label{example1.3}%
Let $\mathcal{I}=\{0,1\}, \varLambda=\varSigma_{\!\mathcal{I}}^+\setminus\{(0,0,0,\ldots), (1,1,1,\ldots)\}$ and let
$\bS\colon\mathcal{I}\rightarrow\mathbb{C}^{2\times 2}$ be defined by
\begin{equation*}
0\mapsto S_0=\left[\begin{matrix}1&0\\ 0&0\end{matrix}\right],\quad 1\mapsto S_1=\left[\begin{matrix}0&0\\ 0&1\end{matrix}\right].
\end{equation*}
It is easily seen that $\rho(\bS)=1$ and $\bS$ is $\varLambda$-absolutely
asymptotically stable. However, $\rho(\bS_{\upharpoonright\!\varLambda})=1$. Moreover, for any $N\ge1$, one can find some $i(\cdot)=(i_n)_{n=1}^{+\infty}\in\varLambda$ such that
$\rho(S_{i_N}\cdots S_{i_1})=1$.
Note here that $\varLambda$ is $\theta_+$-invariant, but it is an open and noncompact subset of $\varSigma_{\!\mathcal{I}}^+$.
\end{example}

\begin{rem}\label{rem1.4}%
To any $\varepsilon>0$, there always exists a norm
$|\pmb{|}\cdot\pmb{|}|_\varepsilon$ on $\mathbb{C}^d$ such that
\begin{equation*}
|\pmb{|}S_i\pmb{|}|_\varepsilon\le\hat{\rho}(\bS)+\varepsilon\quad\forall i\in\mathcal{I},
\end{equation*}
for example in \cite{RS}, also see \cite{El, OR97, SWP} for much shorter proofs. This implies that
\begin{equation*}
\hat{\rho}(\bS)=\inf_{\|\cdot\|\in\mathcal{N}}\left\{\sup_{i\in\mathcal{I}}\|S_i\|\right\},
\end{equation*}
where $\mathcal{N}$ denotes the set of all possible vector norms on $\mathbb{C}^d$.

So, whenever $\hat{\rho}(\bS)<1$
one always can pick a pre-extremal norm $|\pmb{|}\cdot\pmb{|}|$ on
$\mathbb{C}^d$ so that there exists a constant $\hat{\gamma}$ with
\begin{equation*}
|\pmb{|}S_i\pmb{|}|\le\hat{\gamma}<1   \quad\forall i\in\mathcal{I}.\leqno{(\star)}
\end{equation*}
Thus, $\|S_{i_n}\cdots S_{i_1}\|\to 0$ as $n\to\infty$ uniformly for
$i(\cdot)\in\varSigma_{\!\mathcal {I}}^+$
whenever $\hat{\rho}(\bS)<1$. However, this inequality $(\star)$ is not, in general, the case
for the constrained case $\hat{\rho}(\bS_{\upharpoonright\!\varLambda})<1$ when
$\varLambda\not=\varSigma_{\!\mathcal {I}}^+$ because of the lack of the semigroup structure of $\bS_{\upharpoonright\!\varLambda}^+$ as mentioned before. In fact, the $\varLambda$-stability of $\bS$ cannot imply the stability of every subsystems.
This point causes an essential
difference between the case free of any switching constraints and
one obeying switching constraints.
\end{rem}

\begin{rem}\label{rem1.5}%
For the case free of constraints, there holds the following identity:
\begin{equation*}
\rho(\bS)=\sup_{n\ge1}\left\{\sup_{w\in\mathcal{I}^n}\sqrt[n]{\rho(\bS_w)}\right\},\leqno{(*)}
\end{equation*}
which is very important; this is because it simply implies the continuity of $\rho(\bS)$ with respect to $\bS\colon\mathcal{I}\rightarrow\mathbb{C}^{d\times d}$ under the $\mathrm{C}^0$-topology \cite{HS}.
For example, see~\cite[Lemma~3.1]{DL} and \cite[Remark in Section~1]{Bochi}. Moreover, this is
used in~\cite{LW,BT,SWP}. Here we present an other proof for this. Since for any $\varepsilon>0$ one can pick out a norm
$|\pmb{|}\cdot\pmb{|}|_\varepsilon$ on $\mathbb{C}^d$ such that
$|\pmb{|}S_i\pmb{|}|_\varepsilon\le\hat{\rho}(\bS)+\varepsilon$ for all $i\in\mathcal {I}$, as mentioned in
Remark~\ref{rem1.4}. So, from the Berger-Wang formula, it
follows that
\begin{equation*}
\sqrt[n]{\rho(\bS_w)}\le \sqrt[n]{\hat{\rho}(\bS_w)}
\le \rho(\bS)+\varepsilon \qquad\forall w\in\mathcal {I}^n\textrm{ and }
n\ge 1.
\end{equation*}
Thus, $\sup_{w\in \mathcal {I}^n}\sqrt[n]{\rho(\bS_w)}\le\rho(\bS)$ for any $n\ge 1$ and so $\sup_{n\ge1}\left\{\sup_{w\in \mathcal {I}^n}\sqrt[n]{\rho(\bS_w)}\right\}=\rho(\bS)$.

In our situation, however, the above $(*)$ does not need to hold restricted to $\varLambda$ because of the lack of condition $(\star)$. We consider an explicit constrained system. Let $\bS$ be defined as in Example~\ref{example1.3} and let
\begin{equation*}
\varLambda=\{i^\prime(\cdot)=(0,1,0,1,0,1,\ldots),\quad i^{\prime\prime}(\cdot)=(1,0,1,0,1,0,\ldots)\}.
\end{equation*}
Since $\theta_+(i^\prime(\cdot))=i^{\prime\prime}(\cdot)$ and $\theta_+(i^{\prime\prime}(\cdot))=i^\prime(\cdot)$, $\varLambda$ is a $\theta_+$-invariant compact subset of $\varSigma_{\!\mathcal{I}}^+$. Clearly,
\begin{equation*}
\rho(\bS_{\upharpoonright\!\varLambda})=0\lneqq\sup_{n\ge1}\left\{\max_{i(\cdot)\in\varLambda}\sqrt[n]{\rho(S_{i_n}\cdots S_{i_1})}\right\}=1.
\end{equation*}
This shows that the dynamics behavior of a constrained system is sometimes very different from that of a system free of any constraints.
\end{rem}

Similar to the proof of the Berger-Wang formula presented before, it follows easily from Theorem~B
that if $\rho(\bS)<1$ then
$\bS$, free of
constraints, is absolutely exponentially stable. So, this theorem
extends Brayton-Tong \cite[Theorem~4.1]{BT80}, Barabanov~\cite{Bar},
Daubechies-Lagarias~\cite[Theorem~4.1]{DL},
Gurvits~\cite[Theorem~2.3]{Gur} and
Shih-Wu-Pang~\cite[Theorem~1]{SWP} for a discrete-time linear switched system that
is free of any switching constraints to one which obeys some
switching constraints.

Finally, the paper ends with some questions related closely to Theorems A
and B for us to further study in Section~\ref{sec4}.

\section{The Gel'fand-type spectral-radius formula obeying constraints}\label{sec2}%

In this section, we will devote our attention to proving Theorem~A which asserts
a Gel'fand-type spectral-radius formula of a set of matrices obeying some
switching constraints, using ergodic-theoretic approaches.

\subsection{Some ergodic-theoretic results}\label{sec2.1}%
Let $T\colon\varOmega\rightarrow\varOmega$ be a continuous
transformation of a compact topological space $\varOmega$. Let
$\mathscr{B}_{\!\varOmega}$ be the Borel $\sigma$-field of the
space $\varOmega$, which is generated by all open sets of the
topology space $\varOmega$.

\begin{defn}[See \cite{NS}] \label{def2.1}
A probability measure $\mu$ on the Borel measurable space
$(\varOmega,\mathscr{B}_{\!\varOmega})$ is said to be
\emph{$T$-invariant}, write as $\mu\in\mathcal
{M}_{\textit{inv}}(\varOmega,T)$, if $\mu=\mu\circ T^{-1}$, i.e.
$\mu(B)=\mu(T^{-1}(B))$ for all $B\in\mathscr{B}_{\!\varOmega}$. A
$T$-invariant probability measure $\mu$ is called
\emph{$T$-ergodic}, write as $\mu\in\mathcal
{M}_{\textit{erg}}(\varOmega,T)$, provided that for
$B\in\mathscr{B}_{\!\varOmega}$, $\mu\left((B\setminus
T^{-1}(B))\cup(T^{-1}(B)\setminus B)\right)=0$ implies $\mu(B)=1$ or $0$.
\end{defn}

To prove Theorem~A, we need several ergodic-theoretic lemmas. The first is the standard Kingman subadditive ergodic theorem.

\begin{Theorem}[See \cite{Kin}]\label{thm2.2R}
Let $\langle f_n\rangle_{n=1}^{+\infty}\colon\varOmega
\rightarrow\mathbb{R}\cup\{-\infty\}$ be a sequence of upper-bounded Borel measurable
functions such that $f_{m+n}(\omega) \le f_n(T^m(\omega))+f_m(\omega)$
for every $\omega\in\varOmega$ and
any $m,n\ge 1$. Then, for any $\mu\in\mathcal
{M}_{\textit{erg}}(\varOmega,T)$, it holds that
\begin{equation*}
\lim_{n\to+\infty}\frac{1}{n}\int_{\varOmega}f_n(\omega)\,d\mu(\omega)=\inf_{n\ge1}\frac{1}{n}\int_{\varOmega}f_n(\omega)\,d\mu(\omega)=\lim_{n\to+\infty}\frac{1}{n}f_n(\omega)
\end{equation*}
for $\mu$-a.s. $\omega\in\varOmega$.
\end{Theorem}

As usual, one can introduce a natural topology for $\mathbb{R}\cup\{-\infty\}$ under which $[0,+\infty)$ is homeomorphic to $\mathbb{R}\cup\{-\infty\}$ by a strictly increasing continuous function from $\mathbb{R}\cup\{-\infty\}$ onto $[0,+\infty)$ with $-\infty\mapsto0$.
The second lemma needed is the semi-uniform subadditive ergodic
theorem, independently due to S.\,J.~Schreiber~\cite{Sch} and R.~Sturman and
J.~Stark~\cite{SS}, which could be stated as follows:

\begin{Theorem}[See \cite{Sch,SS}]\label{thm2.3R}
Let $\langle f_n\rangle_{n=1}^{+\infty}\colon\varOmega
\rightarrow\mathbb{R}\cup\{-\infty\}$ be a sequence of continuous
functions such that $f_{\ell+m}(\omega) \le
f_\ell(T^m(\omega))+f_m(\omega)$ for every $\omega\in\varOmega$ and
any $\ell,m\ge 1$. If there is a constant $\pmb{\alpha}$ such that
\begin{equation*}
\lim_{n\to+\infty}\frac{1}{n}\int_{\varOmega}f_n(\omega)\,d\mu(\omega)<\pmb{\alpha}\quad\forall \mu\in\mathcal{M}_{\textit{erg}}(\varOmega,T),
\end{equation*}
then there exists an $N\ge 1$ such that for any $\ell\ge N$,
$\sup_{\omega\in\varOmega}\frac{1}{\ell}f_\ell(\omega)<\pmb{\alpha}$.
\end{Theorem}

See \cite{Dai10} for an elementary and short proof of the above semi-uniformity theorem.
Next, we put
\begin{equation*}
{\chi}(\langle f_n\rangle_1^{\infty})=\lim_{n\to+\infty}\left\{\sup_{\omega\in
\varOmega}\frac{1}{n}f_n(\omega)\right\}\quad\textrm{and}\quad \chi(\mu,\langle f_n\rangle_1^{\infty})=\inf_{\ell\ge
1}\frac{1}{\ell}\int_{\varOmega}f_\ell(\omega)\,d\mu(\omega).
\end{equation*}
Clearly, ${\chi}(\langle f_n\rangle_1^{\infty})\le\max_{\omega\in\varOmega}f_1(\omega)<+\infty$ by the subadditivity and the continuity of $f_n(\omega)$ in $\omega\in\varOmega$.

As a result of Theorem~\ref{thm2.3R}, we can simply obtain the following version of Theorem~\ref{thm2.3R}.

\begin{Lemma}\label{lem2.4R}
Let $\langle
f_n\rangle_1^{+\infty}\colon\varOmega\rightarrow\mathbb{R}\cup\{-\infty\}$ be
be a $T$-subadditive sequence of continuous
functions. Then
\begin{equation*}
{\chi}(\langle
f_n\rangle_1^{\infty})=\max_{\mu\in\mathcal
{M}_{\textit{erg}}(\varOmega,T)}\chi(\mu,\langle f_n\rangle_1^{\infty}).
\end{equation*}
\end{Lemma}

\begin{proof}
Let $\pmb{\alpha}={\chi}(\langle f_n\rangle_1^{\infty})$. It is easy to see $\pmb{\alpha}\ge\chi(\mu,\langle f_n\rangle_1^{\infty})$ from Theorem~\ref{thm2.2R}. To prove the statement, suppose, by contradiction, that $\chi(\mu,\langle f_n\rangle_1^{\infty})<\pmb{\alpha}$ for all $\mu\in\mathcal
{M}_{\textit{erg}}(\varOmega,T)$. Then from Theorem~\ref{thm2.3R}, it follows that there exists an $N\ge1$ such that $\sup_{\omega\in\varOmega}\frac{1}{N}f_N(\omega)<\pmb{\alpha}$. Since $\varOmega$ is compact and $f_N$ is continuous, one can find some constant $\alpha^\prime<\pmb{\alpha}$ such that $\frac{1}{N}f_N(\omega)\le\alpha^\prime$ for all $\omega\in\varOmega$. Combining this with the subadditivity of $\langle
f_n\rangle_1^{+\infty}$ implies that ${\chi}(\langle f_n\rangle_1^{\infty})\le\alpha^\prime$, a contradiction. This proves Lemma~\ref{lem2.4R}.
\end{proof}

We notice here that the compactness of $\varOmega$ is important for
the statements of Theorem~\ref{thm2.3R} and Lemma~\ref{lem2.4R}, but not necessary for Theorem~\ref{thm2.2R}.

We call the numbers ${\chi}(\langle f_n\rangle_1^{\infty})$ and
$\chi(\mu,\langle f_n\rangle_1^{\infty})$, defined above, the \emph{joint growth rate} and \emph{growth rate at $\mu$}, of
the subadditive sequence $\langle f_n\rangle_1^{\infty}$, respectively. In addition, put
\begin{equation*}
\chi(\omega,\langle
f_n\rangle_1^{\infty})=\limsup_{n\to+\infty}\frac{1}{n}f_n(\omega).
\end{equation*}
Then from Theorem~\ref{thm2.2R}, it follows that
\begin{equation*}
\chi(\omega,\langle f_n\rangle_1^{\infty})=\chi(\mu,\langle
f_n\rangle_1^{\infty})  \qquad \mu\textrm{-a.s.}\ \omega\in\varOmega.
\end{equation*}
So, for any $T$-subadditive sequence $\langle f_n\rangle_1^{\infty}$ as in
Theorem~\ref{thm2.3R}, by Lemma~\ref{lem2.4R} we have
\begin{equation*}
{\chi}(\langle
f_n\rangle_1^{\infty})=\max_{\omega\in\varOmega}{\chi}(\omega,\langle
f_n\rangle_1^{\infty}).
\end{equation*}
Thus, we can obtain the following optimization result for the subadditive
function sequence $\langle f_n(\omega)\rangle_1^{\infty}$ given as in Theorem~\ref{thm2.3R}.

\begin{Lemma}\label{lem2.5R}
Let $\langle f_n\rangle_1^\infty$ be arbitrary given as in
Theorem~\ref{thm2.3R}. Then there can be found some $\pmb{\mu}_*\in\mathcal
{M}_{\textit{erg}}(\varOmega,T)$ such that ${\chi}(\langle
f_n\rangle_1^{\infty})=\chi(\pmb{\mu}_*,\langle f_n\rangle_1^{\infty})$. This also
implies that $\chi(\langle f_n\rangle_1^{\infty})=\chi(\omega,\langle
f_n\rangle_1^{\infty})$ for $\pmb{\mu}_*$-a.s. $\omega\in\varOmega$.
\end{Lemma}

This result is an extension of \cite[Theorem~3.1]{DHX-ERA} from finite set $\bS$ to infinite case.
For the case that $\langle f_n\rangle_1^\infty\colon\varOmega\rightarrow\mathbb{R}$, the statement of Lemma~\ref{lem2.5R} can be read in Y.-L.~Cao~\cite{Cao}.

On the growth of the spectral radius, the following result is due to Ian D.~Morris, which has been proved based on the
multiplicative ergodic theorem (cf.~\cite{FK,Ose,FLQ}) using invariant cone.

\begin{Theorem}[See \cite{Mor}]\label{thm2.6R}
Let $T\colon(\varOmega,\mathscr{B}_{\!\varOmega},\mu)\rightarrow(\varOmega,\mathscr{B}_{\!\varOmega},\mu)$
be a measure-preserving continuous transformation of a metrizable topological
space $\varOmega$, and $\mathcal
{L}\colon\varOmega\times\mathbb{Z}_+\rightarrow\mathbb{C}^{d\times
d}$ a Borel measurable linear cocycle driven by $T$, i.e.,
\begin{equation*}
\mathcal {L}(\omega,0)=\mathrm{Id}_{\mathbb{C}^d},\quad\mathcal
{L}(\omega,\ell+m)=\mathcal {L}(T^m(\omega),\ell)\mathcal
{L}(\omega,m)\quad \forall \omega\in\varOmega\textrm{ and }\ell,m\ge 1.
\end{equation*}
If $\int_\varOmega\log^+\|\mathcal
{L}(\omega,1)\|d\mu(\omega)<\infty$ where $\log 0=-\infty$ and $\log^+x=\max\{0,\log x\}$ for any $x\ge0$, then
one can find a $T$-invariant Borel subset $\varUpsilon_{\!\mu}$ of
$\varOmega$ with $\mu(\varUpsilon_{\!\mu})=1$ such that
\begin{equation*}
\limsup_{n\to+\infty}\frac{1}{n}\log\rho(\mathcal {L}(\omega,n))
=\lim_{n\to+\infty}\frac{1}{n}\log\|\mathcal{L}(\omega,n)\|
\end{equation*}
for all $\omega\in\varUpsilon_{\!\mu}$.
\end{Theorem}

Particularly, let $\varOmega=\varSigma_{\!\mathcal{I}}^+, T=\theta_+$ and $\mathcal{L}(\omega,n)=S_{i_n}\dotsm S_{i_1}$ for $\omega=i(\cdot)$. Then, this theorem tells us that there holds:
\begin{equation*}
\limsup_{n\to+\infty}\sqrt[n]{\rho(S_{i_n}\cdots S_{i_1})}=\lim_{n\to+\infty}
\sqrt[n]{\|S_{i_n}\cdots S_{i_1}\|}\quad\mu\textrm{-a.s. }i(\cdot)\in\varSigma_{\!\mathcal{I}}^+,
\end{equation*}
for every $\theta_+$-invariant probability measure $\mu$ on $\varSigma_{\!\mathcal{I}}^+$.
\subsection{Proof of Theorem~A and an optimization result}\label{sec2.2}%
Let
$\varLambda\subset\varSigma_{\!\mathcal{I}}^+$ be a $\theta_+$-invariant
closed set and $\bS\colon \mathcal{I}\rightarrow
\mathbb{C}^{d\times d}$ be continuous. Then, the $\varLambda$-stability of the linear switched system given by
\begin{equation*}
x_n=S_{i_n}\cdots S_{i_1}x_0\qquad (n\ge 1,\ x_0\in\mathbb{C}^d,\;
i(\cdot)\in\varSigma_{\!\mathcal{I}}^+),
\end{equation*}
is equivalent to the stability of the linear cocycle defined as follows:
\begin{equation*}
\mathcal
{L}\colon\varLambda\times\mathbb{Z}_+\rightarrow\mathbb{C}^{d\times
d};\quad (i(\cdot),k)\mapsto\mathcal {L}(i(\cdot),k)=
\begin{cases}
\mathrm{Id}_{\mathbb{C}^d}& \textrm{if }k=0,\\
S_{i_k}\cdots S_{i_1}& \textrm{if }k\ge1.
\end{cases}
\end{equation*}
Under the product topology of $\varSigma_{\!\mathcal {I}}^+$, the
cocycle $\mathcal {L}(i(\cdot),k)$ is continuous, where
$\mathbb{Z}_+=\{0,1,2,\dotsc\}$ is endowed with the discrete topology. In addition, note that $\varSigma_{\!\mathcal {I}}^+$ is metrizable.

Now, we are ready to prove our Gel'fand-type spectral-radius theorem.

\begin{proof}[\textbf{Proof of Theorem~A}]
Since $\varLambda$ is a compact subset and $\mathcal {L}(i(\cdot),1)$ is
continuous with respect to $i(\cdot)\in\varLambda$,
$\log^+\|\mathcal {L}(i(\cdot),1)\|$ is bounded uniformly for
$i(\cdot)\in\varLambda$. Applying Theorem~\ref{thm2.6R} in the
case $\varOmega=\varLambda$ and
$T={\theta_+}_{\!\upharpoonright\!\varLambda}$, we could define a
$\theta_+$-invariant subset $\varUpsilon\subset\varLambda$ such
that $\mu(\varUpsilon)=1$ for all $\mu\in\mathcal
{M}_{\textit{erg}}(\varLambda,
{\theta_+}_{\!\upharpoonright\!\varLambda})$ and that
\begin{equation*}
\limsup_{n\to+\infty}\frac{1}{n}\log\rho(\mathcal
{L}(i(\cdot),n))
=\lim_{n\to+\infty}\frac{1}{n}\log\|\mathcal{L}(i(\cdot),n)\|
\quad\forall i(\cdot)\in\varUpsilon.
\end{equation*}
In fact, for each $\mu\in\mathcal
{M}_{\textit{erg}}(\varLambda,
{\theta_+}_{\!\upharpoonright\!\varLambda})$ we can define a set $\varUpsilon_{\!\mu}$ by Theorem~\ref{thm2.6R} and then let $\varUpsilon=\bigcup\varUpsilon_{\!\mu}$.
Then from the definition
of the generalized spectral radius, there holds the inequality
\begin{equation*}
\rho(\bS_{\upharpoonright\!\varLambda})\ge\limsup_{n\to+\infty}\sqrt[n]{\rho(\mathcal
{L}(i(\cdot),n))}\quad\forall i(\cdot)\in\varUpsilon.
\end{equation*}
Theorem~\ref{thm2.6R} implies that
\begin{equation*}
\rho(\bS_{\upharpoonright\!\varLambda})\ge\lim_{n\to+\infty}\sqrt[n]{\|\mathcal
{L}(i(\cdot),n)\|}\quad\forall i(\cdot)\in\varUpsilon.
\end{equation*}
Since $f_n(i(\cdot))=\log\|\mathcal
{L}(i(\cdot),n)\|$ is continuous with respect to $i(\cdot)\in\varLambda$ and the sequence $\langle f_n\rangle_1^{+\infty}$ is $\theta_+$-subadditive, from Theorem~\ref{thm2.2R} it follows that
\begin{align*}
\log\rho(\bS_{\upharpoonright\!\varLambda})&\ge\inf_{n\ge
1}\left\{\int_{\varLambda}\log\sqrt[n]{\|\mathcal
{L}(i(\cdot),n)\|}\,d\mu(i(\cdot))\right\}\\
&=\lim_{n\to+\infty}\int_{\varLambda}\log\sqrt[n]{\|\mathcal
{L}(i(\cdot),n)\|}\,d\mu(i(\cdot))
\end{align*}
for all $\mu\in\mathcal {M}_{\textit{erg}}(\varLambda, {\theta_+}_{\!\upharpoonright\!\varLambda})$.
Now, applying Theorem~\ref{thm2.3R} one can obtain that
\begin{equation*}
\log\rho(\bS_{\upharpoonright\!\varLambda})\ge\lim_{n\to+\infty}\left\{\sup_{i(\cdot)\in\varLambda}\log\sqrt[n]{\|\mathcal
{L}(i(\cdot),n)\|}\right\}.
\end{equation*}
Thus, from the definition of $\hat{\rho}(\bS_{\upharpoonright\!\varLambda})$ there
holds the inequality $\rho(\bS_{\upharpoonright\!\varLambda})\ge\hat{\rho}(\bS_{\upharpoonright\!\varLambda})$
and further there follows that $\rho(\bS_{\upharpoonright\!\varLambda})=\hat{\rho}(\bS_{\upharpoonright\!\varLambda})$
from $\rho(\bS_{\upharpoonright\!\varLambda})\le\hat{\rho}(\bS_{\upharpoonright\!\varLambda})$.
This completes the proof of Theorem~A.
\end{proof}

As a consequence of Lemma~\ref{lem2.5R} and Theorem A, we could
obtain at once the following optimization result.

\begin{cor}\label{cor2.7R}
Let $\bS\colon\mathcal{I}\rightarrow
\mathbb{C}^{d\times d}$ be continuous and assume
$\varLambda\subset\varSigma_{\!\mathcal {I}}^+$ is an invariant compact
set of the one-sided Markov shift
$\theta_+\colon\varSigma_{\!\mathcal{I}}^+\rightarrow\varSigma_{\!\mathcal{I}}^+$.
Then, for the linear switched system
\begin{equation*}
x_n=S_{i_n}\cdots S_{i_1}x_0\qquad (n\ge 1,\ x_0\in\mathbb{C}^d,\;
i(\cdot)\in\varLambda),
\end{equation*}
there holds that
\begin{equation*}
\rho(\bS_{\upharpoonright\!\varLambda})=\max_{\mu\in\mathcal
{M}_{\textit{erg}}(\varLambda,
{\theta_+}_{\!\upharpoonright\!\varLambda})}\{\exp\chi(\mu,\bS)\}
=\max_{i(\cdot)\in\varLambda}\left\{\exp\chi(\bS_{i(\cdot)})\right\}
=\max_{(x_0,i(\cdot))\in\mathbb{C}^d\times\varLambda}\left\{\exp\chi(x_0,\bS_{i(\cdot)})\right\}.
\end{equation*}
\end{cor}

\noindent Here $\chi(\bS_{i(\cdot)})$ is defined as Section~\ref{sec1.1}, and
\begin{equation*}
\chi(\mu,\bS):=\limsup_{n\to+\infty}\frac{1}{n}\log\|S_{i_n}\cdots S_{i_1}\|
\quad \textrm{for }\mu\textrm{-a.s. }i(\cdot)\in\varLambda
\end{equation*}
is called the (maximal) Lyapunov exponents of $\bS$ at $\mu$.

\begin{proof}
Applying Lemma~\ref{lem2.5R} to the case that
$f_n(i(\cdot))=\log\|S_{i_n}\cdots S_{i_1}\|$
for $i(\cdot)\in\varOmega=\varLambda$ and
$T={\theta_+}_{\!{\upharpoonright\varLambda}}$, one can find some
$\theta_+$-ergodic probability, say $\mu_*$, on $\varLambda$ such that
\begin{equation*}
\hat{\rho}(\bS_{\upharpoonright\!\varLambda})=\exp\chi(\mu,\bS)=\exp\chi(\bS_{i(\cdot)})\quad
\textrm{for }\mu_*\textrm{-a.s. }i(\cdot)\in\varLambda.
\end{equation*}
Furthermore, from the multiplicative ergodic theorem~\cite{FK,Ose}, it follows that there always are unit
vectors $x_0\in\mathbb{C}^d$ satisfying $\chi(\bS_{i(\cdot)})=\chi(x_0,\bS_{i(\cdot)})$.
Thus, the statement follows at once from Theorem~A.
\end{proof}

Thus, there holds the following.

\begin{cor}\label{cor2.8R}
Let $\bS\colon\mathcal{I}\rightarrow
\mathbb{C}^{d\times d}$ be continuous and assume
$\varLambda\subset\varSigma_{\!\mathcal {I}}^+$ is an invariant compact
set of the one-sided Markov shift
$\theta_+\colon\varSigma_{\!\mathcal{I}}^+\rightarrow\varSigma_{\!\mathcal{I}}^+$.
Then, the following statements are equivalent to each other.
\begin{description}
\item[(1)] $\rho(\bS_{\upharpoonright\!\varLambda})<1$.

\item[(2)] $\bS$ is $\varLambda$-absolutely exponentially stable.

\item[(3)] $\bS$ is ``$\varLambda$-pointwise exponentially stable", i.e.,
$\chi(x_0,\bS_{i(\cdot)})<0$ for all $x_0\in\mathbb{C}^d$ and any $i(\cdot)\in\varLambda$.
\end{description}
\end{cor}

This statement will be useful for proving Theorem~B in Section~\ref{sec3}.
\section{Criteria for stability under switching constraints}\label{sec3}%
In this section, we will prove Theorem~B stated in Section~\ref{sec1.2}, using Theorem~A and Corollary~\ref{cor2.8R} that have been proved in Section~\ref{sec2}.
As before, we let $\varSigma_{\!\mathcal {I}}^+$ denote the space of all switching signals $i(\cdot)\colon\mathbb{N}\rightarrow\mathcal{I}$.
Let $\theta_+\colon\varSigma_{\!\mathcal {I}}^+\rightarrow\varSigma_{\!\mathcal {I}}^+$
be the one-sided Markov shift defined as in Theorem~A, that is to say,
\begin{equation*}
\theta_+\colon i(\cdot)\mapsto i(\cdot+1)\quad\forall i(\cdot)=(i_n)_{n=1}^{+\infty}\in\varSigma_{\!\mathcal {I}}^+.
\end{equation*}
Let $\varLambda$ be an arbitrary, $\theta_+$-invariant, closed, and nonempty subset of
$\varSigma_{\!\mathcal{I}}^+$ and
$\bS\colon\mathcal{I}\rightarrow\mathbb{C}^{d\times d}$
continuous with respect to $i\in\mathcal {I}$.
Recall that the
linear switched system with constraint
$\varLambda$
\begin{equation*}
x_n=S_{i_n}\cdots S_{i_1}x_0\qquad
(n\ge 1,\ x_0\in\mathbb{C}^d,\; i(\cdot)\in
\varLambda)\leqno{\bS_{\upharpoonright\!\varLambda}}
\end{equation*}
is called \emph{$\varLambda$-absolutely asymptotically stable} in case
\begin{equation*}
S_{i_n}\cdots S_{i_1}\to \mathbf{0}_{d\times d}\; \textrm{ as }
n\to\infty\qquad\forall i(\cdot)=(i_n)_{n=1}^{+\infty}\in\varLambda,
\end{equation*}
where $\mathbf{0}_{d\times d}$ is the origin of $\mathbb{C}^{d\times d}$.
Let $\|\cdot\|_2$ be the matrix norm on $\mathbb{C}^{d\times d}$ induced by the usual Euclidean vector norm on $\mathbb{C}^d$.

\subsection{A criterion of $\varLambda$-stability}%
First, we present a criterion of $\varLambda$-absolute asymptotic stability (Lemma~\ref{lem3.1}), which is an extension of
\cite[Theorem~4.1]{BT80} from the case free of any constraints to a system which obeys switching constraints.

\begin{Lemma}\label{lem3.1}
Let $\varLambda$ be a $\theta_+$-invariant compact subset of $\varSigma_{\!\mathcal{I}}^+$ and let
\begin{equation*}
\bS_{\upharpoonright\!\varLambda}^+(0)=\left\{\mathrm{Id}_{\mathbb{C}^d}\right\},\quad \textsl{\textbf{S}}_{\upharpoonright\!\varLambda}^+(\ell)=\left\{S_{i_\ell}\cdots S_{i_1};\
i(\cdot)\in\varLambda\right\}\ \textrm{for }\ell\ge1
\quad\textrm{and}\quad\bS_{\upharpoonright\!\varLambda}^+={\bigcup}_{\ell\ge0}\bS_{\upharpoonright\!\varLambda}^+(\ell).
\end{equation*}
Then, $\bS$ is $\varLambda$-absolutely asymptotically stable if and only if
\begin{description}
\item[(1)] $\bS_{\upharpoonright\!\varLambda}^+$ is bounded in
$\mathbb{C}^{d\times d}$,i.e., $\exists\,\beta>0$ such that
$\|A\|_2\le\beta\ \forall A\in\bS_{\upharpoonright\!\varLambda}^+$; and

\item[(2)] there exists a constant
$\gamma>0$ and an integer $N\ge1$ such that
\begin{equation*}
\rho(A)\le\gamma<1\quad\forall A\in\bS_{\upharpoonright\!\varLambda}^+(\ell),
\end{equation*}
for any $\ell\ge N$.
\end{description}
\end{Lemma}

The condition (1) in Theorem~\ref{lem3.1} means that
$\bS$ is Lyapunov stable restricted to $\varLambda$.
This theorem is itself very interesting and it is a key step towards
the proof of Theorem~B. Comparing to the case that is free of any switching
constraints, now $\bS_{\upharpoonright\!\varLambda}^+$ is not a
semigroup. This might cause an essential difficulty described as
follows: if $\varLambda=\varSigma_{\!\mathcal{I}}^+$, i.e., free of any switching constraints, then condition (1) above implies that there can be found
a pre-extremal vector norm $|\pmb{|}\cdot\pmb{|}|$ on $\mathbb{C}^d$ for $\bS$ such that $|\pmb{|}A\pmb{|}|\le 1$ for all $A\in\bS_{\upharpoonright\!\varLambda}^+$; But now in our context, this does not need to be true.

We note here that if the joint spectral radius
$\hat{\rho}(\bS_{\upharpoonright\!\varLambda})<1$ then
$\bS$ is obviously
$\varLambda$-absolutely asymptotically stable from Corollary~\ref{cor2.8R}. In fact, there holds the
following stronger result.

\begin{Lemma}\label{lem3.2}
Let $\varLambda$ be a $\theta_+$-invariant compact subset of
$\varSigma_{\!\mathcal{I}}^+$. Then
$\hat{\rho}(\bS_{\upharpoonright\!\varLambda})<1$ if and only
if $\bS$ is ``$\varLambda$-uniformly exponentially stable";
that is, there exists a number $0<\lambda<1$ and an integer $N\ge 1$ such that
\begin{equation*}
\|S_{i_n}\cdots S_{i_1}\|_2\le\lambda^n\quad
\forall i(\cdot)\in\varLambda\textrm{ and }n\ge N.
\end{equation*}
\end{Lemma}

\begin{proof}
Let $1>\lambda>\hat{\rho}(\bS_{\upharpoonright\!\varLambda})$. Then from the
definition of $\hat{\rho}(\bS_{\upharpoonright\!\varLambda})$, there is some integer $N\ge 1$
such that
\begin{equation*}
\sup_{i(\cdot)\in\varLambda}\sqrt[n]{\|S_{i_n}\cdots S_{i_1}\|_2}\le\lambda
\quad \forall n\ge N.
\end{equation*}
So, $\bS$ is $\varLambda$-uniformly
exponentially stable. Conversely, if there exists a constant
$0<\lambda<1$ and an integer $N\ge 1$ such that
\begin{equation*}
\|S_{i_n}\cdots S_{i_1}\|_2\le\lambda^n\quad
\forall i(\cdot)\in\varLambda\textrm{ and }n\ge N,
\end{equation*}
then
\begin{equation*}
\hat{\rho}(\bS_{\upharpoonright\!\varLambda})=\inf_{n\ge
1}\left\{\sup_{i(\cdot)\in\varLambda}\sqrt[n]{\|S_{i_n}\cdots S_{i_1}\|_2}\right\}\le\lambda<1,
\end{equation*}
as desired.
This proves Lemma~\ref{lem3.2}.
\end{proof}

At the first glance, $\varLambda$-absolute asymptotic stability is weaker than the $\varLambda$-absolute
exponential stability for the switching system $\bS$. However, they are
equivalent to each other as is shown in the case free of any switching
constraints (cf.~\cite[Theorem~4.1]{DL} and
\cite[Theorem~2.3]{Gur}). In fact, the $\varLambda$-absolute asymptotic stability is equivalent to the $\varLambda$-uniform exponential stability from the Fenichel uniformity theorem~\cite{Fen}, stated as follows:

\begin{Lemma}[See N.~Fenichel~\cite{Fen}]\label{lem3.3}
Let $\varLambda$ be a $\theta_+$-invariant compact subset of $\varSigma_{\!\mathcal{I}}^+$. Then,
$\bS$ is $\varLambda$-absolutely asymptotically stable if and only if it is
$\varLambda$-uniformly exponentially stable.
\end{Lemma}

\begin{rem}
For Lemma~\ref{lem3.3}, the hypothesis that $\varLambda$ is
``compact" is important, as shown by Example~\ref{example1.3} in Section~\ref{sec1.2}.
\end{rem}

Now, we can readily prove Lemma~\ref{lem3.1} using the statements of
Lemmas~\ref{lem3.2} and \ref{lem3.3}.

\begin{proof}[\textbf{Proof of Lemma~\ref{lem3.1}}]
If $\bS$ is $\varLambda$-absolutely asymptotically stable, then from
Lemmas~\ref{lem3.3} and \ref{lem3.2} there follows that
conditions (1) and (2) in
Lemma~\ref{lem3.1} are trivially fulfilled.
Next, let conditions (1) and (2) in Lemma~\ref{lem3.1} both hold. We proceed to prove that
$\bS$ is $\varLambda$-absolutely asymptotically stable.

Assume, by contradiction, that $\bS$
were not $\varLambda$-absolutely asymptotically stable; then one can find some switching signal, say
$i(\cdot)=(i_n)_{n=1}^{+\infty}$, in $\varLambda$ such that
$\|S_{i_n}\cdots S_{i_1}\|_2\not\to 0$ as $n\to\infty$. Using the
boundedness of $\bS_{\upharpoonright\!\varLambda}^+$ in $\mathbb{C}^{d\times d}$, we can pick out an increasing positive integer sequence, say
$\{j_\ell\}_{\ell=1}^{+\infty}$, with $j_\ell\to+\infty$ as $\ell\to+\infty$, such that
\begin{equation*}
C_\ell:=S_{i_{j_\ell}}\cdots S_{i_1}\to C\not=\textbf{0}_{d\times d}\quad \textrm{as }\ell\to\infty.
\end{equation*}
Now, define $B_\ell:=S_{i_{j_{\ell+1}}}\cdots S_{i_{j_\ell+1}}$ and
so $C_{\ell+1}=B_\ell C_\ell$. Since $\theta_+^{j_\ell}(i(\cdot))=i(\cdot+j_\ell)\in\varLambda$, i.e.,
$(i_{n+j_\ell})_{n=1}^{+\infty}$ lies in $\varLambda$, by the
$\theta_+$-invariance of $\varLambda$, one could obtain
$B_\ell\in \bS_{\upharpoonright\!\varLambda}^+$.
Using the boundedness again, we can pick out a subsequence
\begin{equation*}
B_{\ell_k}\to B\in\mathbb{C}^{d\times d}\quad \textrm{as }k\to\infty.
\end{equation*}
Then, $C=BC,\ C\not=\textbf{0}_{d\times d}$, and $\rho(B)=\lim_{k\to\infty}\rho(B_{\ell_k})\le\gamma<1$ by condition (2) of Lemma~\ref{lem3.1}. But
\begin{equation*}
B(\mathrm{Im }C)=\mathrm{Im }C\not=\{\mathbf{0}\},
\end{equation*}
so $B_{\upharpoonright\mathrm{Im }C}$ is the identity. Thus,
$\rho(B)\ge 1$; it is a contradiction to condition (2).

This therefore proves the statement of Lemma~\ref{lem3.1}.
\end{proof}

\subsection{A reduction lemma}%
To prove Theorem~B stated in Section~\ref{sec1.2}, we need an important reduction theorem, which is due to
L.~Elsner~\cite[Lemma~4]{El} and simply proved in X.~Dai~\cite{Dai-JMAA}.

\begin{Lemma}[See \cite{El}]\label{lem3.5}
If $\hat{\rho}(\bS)=1$ and $\bS$ is product unbounded in
$\mathbb{C}^{d\times d}$, then there is a nonsingular $P\in\mathbb{C}^{d\times d}$ and $1\le
d_1<d$ such that
\begin{equation*}
P^{-1}S_i P=\left[\begin{array}{ll}S_i^{(2)}&\clubsuit_i\\
\mathbf{0}_{d_1\times (d-d_1)}&S_i^{(1)}\end{array}\right]\quad\forall  i\in\mathcal{I},
\end{equation*}
where $S_i^{(1)}\in\mathbb{C}^{d_1\times d_1}$.
\end{Lemma}

Here $\bS$ is said to be product unbounded, if the multiplicative
semigroup $\bS^+$ defined in the manner as in Lemma~\ref{lem3.1} in the case
$\varLambda=\varSigma_{\!\mathcal {I}}^+$ is unbounded in
$\mathbb{C}^{d\times d}$ under an arbitrary induced operator norm.
\subsection{Proof of Theorem~B}%
Let $\varLambda\subset\varSigma_{\!\mathcal {I}}^+$ be an invariant compact
set of the one-sided Markov shift
$\theta_+\colon\varSigma_{\!\mathcal {I}}^+\rightarrow\varSigma_{\!\mathcal {I}}^+$, which
gives rise to the constrained linear switched system
\begin{equation*}
x_n=S_{i_n}\cdots S_{i_1}x_0\qquad (n\ge1,\
x_0\in\mathbb{C}^d,\;i(\cdot)\in\varLambda),
\leqno{\bS_{\upharpoonright\!\varLambda}}
\end{equation*}
where $\bS\colon\mathcal{I}\rightarrow\mathbb{C}^{d\times d};\ i\mapsto S_i$ is as in the assumption of Theorem~B.

We now proceed to prove Theorem~B.

\begin{proof}[\textbf{Proof of Theorem~B}]
Clearly, $\mathrm{(a)}\Rightarrow\mathrm{(b)}$ follows from Lemma~\ref{lem3.3} and Corollary~\ref{cor2.8R}, and
$\mathrm{(b)}\Rightarrow \mathrm{(c)}$ follows from Theorem~A,
Corollary~\ref{cor2.8R} and Lemma~\ref{lem3.1}.

So, to prove Theorem~B, we need to prove only $\mathrm{(c)}\Rightarrow \mathrm{(a)}$. According to the definition
of $\rho(\bS_{\upharpoonright\!\varLambda})$ and from Theorem~A, condition $\mathrm{(c)}$ implies that
\begin{equation*}
\rho(\bS_{\upharpoonright\!\varLambda})=\hat{\rho}(\bS_{\upharpoonright\!\varLambda})\le 1.
\end{equation*}
If $\hat{\rho}(\bS_{\upharpoonright\!\varLambda})<1$, then from Lemma~\ref{lem3.2}, there holds condition $\mathrm{(a)}$. So,
we proceed, by induction on the dimension $d$ of the state-space $\mathbb{C}^d$, to prove $\hat{\rho}(\bS_{\upharpoonright\!\varLambda})<1$.

Note that if $\bS_{\upharpoonright\!\varLambda}^+$, defined
as in Lemma~\ref{lem3.1}, is bounded in $\mathbb{C}^{d\times d}$,
then condition $\mathrm{(a)}$ follows from Lemma~\ref{lem3.1} together with
condition $\mathrm{(c)}$. So, the assertion is true for $d=1$; this is because
$\rho(S_{i_n}\cdots S_{i_1})=\|S_{i_n}\cdots S_{i_1}\|_2=|S_{i_n}|\cdots|S_{i_1}|\le\gamma<1$ for any $n\ge N$,
for any $i(\cdot)\in\varLambda$ in this case.

Let $m\ge 1$ be an arbitrarily given integer. Assume the assertion is true for all dimensions $d\le m$.
We claim that the assertion holds for $d=m+1$.

Suppose, by contradiction, that $\hat{\rho}(\bS_{\upharpoonright\!\varLambda})=1$ for dimension $d=m+1$. If
$\bS_{\upharpoonright\!\varLambda}^+$ is bounded in $\mathbb{C}^{(m+1)\times(m+1)}$, by Lemma~\ref{lem3.1} and
condition $\mathrm{(c)}$, $\bS_{\upharpoonright\!\varLambda}$ is $\varLambda$-absolutely asymptotically stable so that
$\hat{\rho}(\bS_{\upharpoonright\!\varLambda})<1$ from Lemmas~\ref{lem3.3} and
\ref{lem3.2}, a contradiction. Therefore $\bS_{\upharpoonright\!\varLambda}^+$
is unbounded in $\mathbb{C}^{(m+1)\times(m+1)}$ and further $\bS$ is product unbounded in $\mathbb{C}^{(m+1)\times(m+1)}$. Then from
Lemma~\ref{lem3.5}, one can find a nonsingular $P\in\mathbb{C}^{(m+1)\times(m+1)}$ and $1\le n_1\le m$ such that
\begin{equation*}
P^{-1}S_i P=\left[\begin{matrix}S_i^{(2)}&\clubsuit_i\\
\mathbf{0}&S_{i}^{(1)}\end{matrix}\right]\quad \forall i\in\mathcal{I},
\end{equation*}
where $S_i^{(1)}\in\mathbb{C}^{n_1\times n_1}$ and $\mathbf{0}$ is the origin
of $\mathbb{C}^{n_1\times(m+1-n_1)}$. Set
\begin{equation*}
\bS^{(r)}=\left\{S_i^{(r)}\,|\,i\in\mathcal{I}\right\},\quad r=1,2.
\end{equation*}
Then, by condition $\mathrm{(c)}$
\begin{equation*}
\rho(A)\le\gamma<1\quad \forall A\in
  {\bS_{\upharpoonright\!\varLambda}^{(r)}}^+\quad \textrm{for }r=1,2,
\end{equation*}
where ${\bS_{\upharpoonright\!\varLambda}^{(r)}}^+$ is defined similarly to
$\bS_{\upharpoonright\!\varLambda}^+$ based on $\bS_{\upharpoonright\!\varLambda}^{(r)}$. As the
switched systems $\bS_{\upharpoonright\!\varLambda}^{(r)}$ have dimension less than
$m+1$ for $r=1,2$, by the induction assumption and Theorem~A
\begin{equation*}
  \rho\left(\bS_{\upharpoonright\!\varLambda}^{(r)}\right)<1\quad \textrm{for }r=1,2.
\end{equation*}
Therefore
\begin{equation*}
\rho(\bS_{\upharpoonright\!\varLambda})=\max\left\{\rho\left(\bS_{\upharpoonright\!\varLambda}^{(1)}\right),
\rho\left(\bS_{\upharpoonright\!\varLambda}^{(2)}\right)\right\}<1,
\end{equation*}
and $\hat{\rho}(\bS_{\upharpoonright\!\varLambda})<1$ by Theorem~A, contradicting the hypothesis that $\hat{\rho}(\bS_{\upharpoonright\!\varLambda})=1$.

This contradiction shows that $\hat{\rho}(\bS_{\upharpoonright\!\varLambda})<1$, completing the proof of Theorem~B.
\end{proof}

\section{Concluding remarks and further questions}\label{sec4}%

In this paper, using ergodic theory we have studied the relationship of the joint spectral radius and the generalized spectral radius of a
linear switched system obeying some type of switching constraints, and presented several stability criteria.
We now raise some questions to further study.

Theorem~A asserts a Gel'fand-type spectral-radius formula for a linear switched
system obeying some switching constraints. Let
$\varLambda\subsetneq\varSigma_{\!\mathcal {I}}^+$ be an invariant closed
set of the one-sided Markov shift $\theta_+$. Clearly, for any
$i(\cdot)=(i_n)_{n=1}^{+\infty}\in\varLambda$ and any $n\ge1$,
the sub-word $w=(i_1,\ldots,i_n)$ of length $n$ does not need to be extended to a permissive periodic switching signal, i.e., although
\begin{equation*}
(\overset{w}{\overbrace{i_1,\ldots,i_n}},\overset{w}{\overbrace{i_1,\ldots,i_n}},\ldots)\in\varSigma_{\!\mathcal{I}}^+
\end{equation*}
is a periodic point of $\theta_+$, but it need not belong to the
given subset $\varLambda$. For any $n\ge 1$, put
\begin{equation*}
W_{\!\mathrm{per}}^n(\varLambda)=\left\{w=(i_1,\ldots,i_n)\in\mathcal
{I}^n\,|\,(\overbrace{i_1,\ldots,i_n},\overbrace{i_1,\ldots,i_n},\ldots)\in\varLambda\right\},
\end{equation*}
called the set of all $\varLambda$-periodic words of length $n$. It
is natural to ask the following question:

\begin{que}\label{que1}
If the periodical switching signals are dense in $\varLambda$ then, does
there hold the following equality:
\begin{equation*}
\limsup_{n\to+\infty}\left\{\sup_{w\in
W_{\!\mathrm{per}}^n(\varLambda)}\sqrt[n]{\rho(\bS_w)}\right\}=
\limsup_{n\to+\infty}\left\{\sup_{w\in
W_{\!\mathrm{per}}^n(\varLambda)}\sqrt[n]{\|\bS_w\|}\right\}\textrm{?}
\end{equation*}
\end{que}
\noindent Here $\bS_w=S_{i_n}\cdots S_{i_1}$ for any word
$w=(i_1,\ldots,i_n)$ of length $n\ge 1$ as before.

In our proof of Theorem~A, the compactness of $\varLambda$ plays a role. So, we naturally ask the following question:

\begin{que}\label{que2}
If $\varLambda$ is a $\theta_+$-invariant closed subset of
$\varSigma_{\!\mathcal {I}}^+$ not necessarily compact, does the
statement of Theorem~A still hold when
$\bS=\{S_i\}_{i\in\mathcal{I}}$ is bounded in $\mathbb{C}^{d\times d}$?
\end{que}

In the statement of Theorem~B, from the results proved in
Section~\ref{sec3} there can still be deduced without the assumption
$\rho(\bS)=1$ that
$\mathrm{(a)}\Leftrightarrow\mathrm{(b)}\Rightarrow\mathrm{(c)}$.
This assumption imposed there is used in the proof of $\mathrm{(c)}\Rightarrow\mathrm{(a)}$
where we need to employ Lemma~\ref{lem3.5}.

So, we ask the following question:

 \begin{que}\label{que3}
 Does the statement of Theorem~B still hold without the assumption
 $\rho(\bS)=1$?
 \end{que}
\noindent Furthermore, we believe that it is very possible to have a
positive solution to Question~\ref{que3}.

\section*{Acknowledgments}%
The author would like to thank the anonymous reviewers for their insightful comments for
further improving the quality of this paper.





\bibliographystyle{model1a-num-names}
\bibliography{<your-bib-database>}







\end{document}